\documentclass[11pt]{amsart}

\usepackage{amsmath,amsthm,amsfonts,amssymb,bm,graphicx,cprotect}

\usepackage{hyperref}
\hypersetup{colorlinks=true,citecolor=blue,linkcolor=blue,urlcolor=blue,
pdfstartview=FitH, pdfauthor=Valentin Blomer and Gergely
Harcos and Djordje Milicevic, pdftitle=Bounds for eigenforms on arithmetic hyperbolic 3-manifolds}

\theoremstyle{plain}
\newtheorem{theorem}{Theorem}
\newtheorem{lemma}{Lemma}

\theoremstyle{remark}
\newtheorem{remark}{Remark}

\theoremstyle{definition}
\newtheorem*{acknowledgement}{Acknowledgements}
\newtheorem*{notation}{Notation}

\numberwithin{equation}{section}

\renewcommand{\geq}{\geqslant}
\renewcommand{\leq}{\leqslant}
\renewcommand{\mod}{\,\mathrm{mod}\,}

\DeclareMathOperator{\GL}{GL}
\DeclareMathOperator{\PGL}{PGL}
\DeclareMathOperator{\SL}{SL}
\DeclareMathOperator{\PSL}{PSL}
\DeclareMathOperator{\SO}{SO}
\DeclareMathOperator{\U}{U}
\DeclareMathOperator{\SU}{SU}
\DeclareMathOperator{\Z}{Z}
\DeclareMathOperator{\M}{M}
\DeclareMathOperator{\sgn}{sgn}
\DeclareMathOperator{\vol}{vol}
\DeclareMathOperator{\re}{Re}
\DeclareMathOperator{\im}{Im}
\DeclareMathOperator{\imm}{im}
\DeclareMathOperator{\sech}{sech}
\DeclareMathOperator*{\res}{res}

\newcommand{\eps}{\varepsilon}
\newcommand{\RR}{\mathbb{R}}
\newcommand{\CC}{\mathbb{C}}
\newcommand{\QQ}{\mathbb{Q}}
\newcommand{\ZZ}{\mathbb{Z}}

\newcommand{\HH}{\mathbb{H}}
\newcommand{\FF}{\mathcal{F}}
\newcommand{\LL}{\mathcal{L}}
\newcommand{\ov}[1]{\overline{#1}}
\newcommand{\OO}{\mathcal{O}_K}

\begin{document}

\author{Valentin Blomer}
\author{Gergely Harcos}
\author{Djordje Mili\'cevi\'c}

\address{Mathematisches Institut, Bunsenstr. 3-5, D-37073 G\"ottingen, Germany} \email{blomer@uni-math.gwdg.de}
\address{MTA Alfr\'ed R\'enyi
Institute of Mathematics, POB 127, Budapest H-1364, Hungary}\email{gharcos@renyi.hu}
\address{Central European University, Nador u. 9, Budapest H-1051, Hungary}\email{harcosg@ceu.hu}
\address{Bryn Mawr College, Department of Mathematics, 101 North Merion Avenue, Bryn Mawr, PA 19010, U.S.A.}\email{dmilicevic@brynmawr.edu}

\title{Bounds for eigenforms on arithmetic hyperbolic 3-manifolds}

\dedicatory{Dedicated to Peter Sarnak on the occasion of his sixtieth birthday}

\thanks{First author supported by the Volkswagen Foundation and ERC Starting Grant~258713. Second author supported by OTKA grants K~101855 and K~104183 and ERC Advanced Grant~228005.}


\begin{abstract}
On a family of arithmetic hyperbolic $3$-manifolds of squarefree level, we
prove an upper bound for the sup-norm of Hecke--Maa{\ss} cusp forms, with a
power saving over the local geometric bound simultaneously in the Laplacian
eigenvalue and the volume. By a novel combination of diophantine
and geometric arguments in a noncommutative setting, we obtain bounds as
strong as the best corresponding results on arithmetic surfaces.
\end{abstract}


\setcounter{tocdepth}{2}

\maketitle

\section{Introduction}

The eigenfunctions of the Laplace operator on a Riemannian manifold and their limiting behavior play a central role in such diverse fields as harmonic and global analysis, thermodynamics, spectral geometry, or quantum mechanics. The distribution of mass of eigenfunctions has been extensively studied in particular with respect to the Quantum Unique Ergodicity Conjecture, bounds for $L^p$-norms, or restriction problems. Much of the recent exciting progress on these questions has been in the case of arithmetic manifolds, where the eigenfunctions respecting also the arithmetic symmetries arise from automorphic forms.

This paper presents bounds on $L^{\infty}$-norms of automorphic forms, which may be regarded as a special case of the restriction problem (where the cycle is reduced to a single point). Good sup-norm bounds in terms of the properties of the automorphic form (for instance its Laplacian eigenvalue) or properties of the underlying symmetric space (for instance its volume) have diverse arithmetic and analytic applications; we mention here only bounds for exponential sums with Hecke eigenvalues (see e.g.\ \cite{HM, Te}), bounds for $L$-functions and shifted convolution problems (see e.g.\ \cite{BHa, HM, Ma}), the multiplicity problem (see e.g.\ \cite{Sa} for an illuminating discussion), and control over the zero set or nodal lines of automorphic functions (see e.g.\ \cite{Ru, GRS}).

For a compact Riemannian manifold $X$ and an $L^2$-normalized eigenfunction $\phi$ for the Laplacian of eigenvalue $\lambda$, one has the general bound (see e.g.\ \cite{So, Do})
\begin{equation}
\label{UpperBound}
\| \phi \|_{\infty} \ll_X \lambda^{\nu(X)} \qquad\text{with}\qquad\nu(X)=(\dim X-1)/4,
\end{equation}
which is sharp for $X=S^n$. For hyperbolic manifolds one expects substantially stronger bounds. The first breakthrough was the result of Iwaniec--Sarnak~\cite{IS} that, for (compact and non-compact) \emph{arithmetic} hyperbolic surfaces and eigenfunctions $\phi$ that are simultaneous eigenfunctions of the Hecke algebra, \eqref{UpperBound} holds with any $\nu(X)>5/24$. VanderKam~\cite{VdK} established similar results for Hecke--Laplace eigenfunctions on the sphere.

In this paper, we focus our attention on the harder case of  arithmetic hyperbolic $3$-manifolds and prove strong sup-norm bounds for Hecke--Maa{\ss} cusp forms. Our manifolds are certain (non-compact) congruence quotients of the upper half space $\HH^3$, viewed as a subset of Hamiltonian quaternions with vanishing fourth coordinate. The upper half space comes naturally with a transitive action of $\SL_2(\CC)$ by fractional linear transformations, and the stabilizer of any point is a conjugate of $\SU_2(\CC)$. Hence we can identify
\[\HH^3\cong \SL_2(\CC)\slash\SU_2(\CC)\cong\Z(\CC)\backslash\GL_2(\CC)\slash\U_2(\CC),\]
and our setup is an analogue of the classical case of hyperbolic surfaces, where $\QQ$ is replaced by an imaginary quadratic number field.  The real
quadratic space of signature $(3,1)$ comes with a natural faithful action of
$\PSL_2(\CC)$, and the image of $\PSL_2(\CC)$ in $\SO_{3,1}(\RR)$ turns out to be the
connected component of the identity \cite[pp.~149--150]{Sh}.
From this point of view, one can regard the present work as a first step towards bounding the sup-norm of automorphic forms on higher-dimensional hyperbolic spaces $\HH^n = \SO^+_{n,1}(\RR)/\SO_n(\RR)$.\\

The large eigenvalue limit is of central importance for the correspondence principle of quantum mechanics. In a classical dynamical system on a manifold $X$, the state (position and momentum) of a particle is described by points in the unit tangent bundle $T^{\ast}X$, with classical observables being functions $a:T^{\ast}X\to\mathbb{C}$; corresponding to this in a quantized system are quantum states $u\in L^2(X)$ and quantum observables, linear operators $\text{Op}(a)$ on $L^2(X)$. The correspondence principle asks for a relation between the classical mechanics and quantum mechanics in the semiclassical limit $\hbar\to 0$, which, in the case of classical Hamiltonian dynamics given by the geodesic flow on $T^{\ast}X$, is the large eigenvalue limit $\lambda\to\infty$. In particular, it has been suggested that the eigenfunctions of a classically ergodic system (such as the geodesic flow on negatively curved manifolds) might be expected to exhibit random patterns of interference extrema and thus rather temperate intensity fluctuations when the wavenumber goes to infinity, as compared to the integrable case. Hence the expectation that a stronger version of \eqref{UpperBound} holds in this context and, in particular, the ``subconvexity conjecture'' that on arithmetic hyperbolic manifolds \eqref{UpperBound} holds with some $\nu(X)<(\dim X-1)/4$.

The sup-norm problem on arithmetic hyperbolic 3-manifolds is substantially more refined and presents a number of new features indicative of what one might expect in the case of higher dimension and rank. In \cite{MR}, Maclachlan and Reid address the question of identifying those arithmetic hyperbolic 3-manifolds that contain immersed totally geodesic surfaces and find a striking dichotomy in terms of the underlying invariant trace field and invariant trace quaternion algebra: an arithmetic hyperbolic 3-manifold either contains no such immersed surfaces, or else it contains infinitely many incommensurable immersed totally geodesic surfaces, all of them arithmetic. The congruence subgroups $\Gamma_0(N)$ of the Bianchi group (see \eqref{Gamma0N}) give rise to prototypical non-compact arithmetic 3-manifolds of Maclachlan--Reid type. It has been shown in \cite{Mi} (see also \cite{RS} for the first result in a compact setting) that on any arithmetic hyperbolic 3-manifold of Maclachlan--Reid type there exists an infinite orthonormal family of cusp forms $\phi$ with the \emph{lower} bound
\[\| \phi \|_{\infty} \gg_{X,\eps} \lambda^{1/4-\eps}.\]
While such a lower bound does not contradict the equidistribution of mass on the level of the Quantum Unique Ergodicity Conjecture (QUE), it does falsify the na\"\i ve expectation that \eqref{UpperBound} might hold with an arbitrary small $\nu(X)>0$ (as would be suggested by the random wave model). From the point of view of the correspondence principle, the coincidence between the discrete layers of power growth and the immersed geometric features is extremely intriguing. On the other hand, in proving the lower bound by leveraging the power of the Hecke algebra, the Maclachlan--Reid algebraic structure manifests itself in the fact that a dense family of special points on our 3-manifolds have large rational stabilizers. This feature is a major obstacle in the upper-bound problem that we must overcome.

The complexity of the chaotic system formed by the quantum eigenstates on an arithmetic manifold can increase
in other directions in addition to the large eigenvalue limit. For example, the pushforwards to the modular surface of Hecke--Maa{\ss} cusp forms of increasing level also equidistribute in mass, in the spirit of QUE~\cite{Ne,NPS}. For a high degree cover $X$ of a fixed manifold (such as a congruence cover of the modular surface), of particular interest is the dependence of the bound \eqref{UpperBound} on $X$, or equivalently on the underlying discrete group. Strong dependence on $X$  was achieved for the first time in \cite{BHo} for congruence covers $\Gamma_0(N) \backslash \HH$ of the modular surface, in the case of squarefree $N$. This result was refined considerably in \cite{HT1,HT2}, leading to the state-of-the-art hybrid bound
\[\| \phi \|_{\infty} \ll_{\eps} \lambda^{5/24}N^{-1/6}(\lambda N)^\eps,\]
established by Templier~\cite{Te}, for Hecke--Maa{\ss} cuspidal newforms of squarefree level $N$.

Arithmetic hyperbolic 3-manifolds arising from congruence subgroups of large level are particularly interesting from a geometric point of view. By Mostow--Prasad rigidity, the volume of a hyperbolic 3-manifold is a topological invariant that can, in arithmetic cases, be explicitly expressed in terms of the arithmetic data and, for congruence covers in particular, in terms of the level of the underlying congruence Kleinian group. Hecke operators, the principal tool for arithmetic improvements on \eqref{UpperBound}, have their counterparts acting on the cohomology of these arithmetic hyperbolic 3-manifolds. For covers of large level, this action and the same algebraic structure that is responsible for large stabilizers underlie a number of geometric features including strong forms of the Virtual Infinite Betti Number Conjecture~\cite{GS,KS}, towers of hyperbolic rational homology 3-spheres with arbitrarily large injectivity radius that are related to the Virtual Haken Conjecture~\cite{CD}, and many more (see e.g. \cite{Sen}).

Our principal results are strong upper bounds for Hecke--Maa{\ss} cusp forms on congruence arithmetic hyperbolic 3-manifolds both in terms of the eigenvalue and the volume of the manifold. We proceed to describe them in more detail.\\

Denoting by $\{1,i,j,k\}$ the usual basis of Hamiltonian quaternions, we shall write a typical point $P \in \HH^3$ as $P = z + rj$ with $r>0$ and $z = x + yi \in \CC$. Correspondingly, we introduce the notation
\begin{equation}\label{parts}
\Im(P):=r,\qquad \re(z):=x, \qquad \im(z):=y.
\end{equation}

Let $K$ be the Gaussian number field $\QQ(i)$ with ring of integers $\OO = \ZZ[i]$. The group $\SL_2(\CC)$ acts on $\HH^3$ by orientation-preserving isometries as in \eqref{action}, and $\SL_2(\OO)$ is a discrete subgroup of it: essentially the Bianchi group associated with $K$. In this paper, we focus on the congruence subgroups
\begin{equation}
\label{Gamma0N}
\Gamma_0(N) := \left\{\begin{pmatrix} a & b\\c & d\end{pmatrix} \in \SL_2(\OO) : N \mid c\right\}
\end{equation}
for \emph{squarefree} Gaussian integers $N \in \ZZ[i]$. The quotient $\Gamma_0(N) \backslash \HH^3$ is a non-compact arithmetic hyperbolic $3$-manifold  with   volume
\[V := \int_{\Gamma_0(N) \backslash \HH^3}  \frac{dx \, dy \, dr}{r^3}   \asymp  [\SL_2(\OO) : \Gamma_0(N)] = |N|^{2 + o(1)}.\]

The Riemannian manifold $\Gamma_0(N) \backslash \HH^3$  is acted on by the Laplace operator $r^2(\partial^2_x + \partial^2_y + \partial_r^2) - r\partial_r$  and the entire family of Hecke operators $T_n$, indexed by $n \in  \OO$ with $(n, N) = 1$, that we describe precisely in the next section. These operators are self-adjoint and commute with each other, and so a basis can be chosen of the cuspidal part of $L^2(\Gamma_0(N)\backslash\HH^3)$ consisting of their joint eigenfunctions. We consider Hecke--Maa{\ss} cuspidal newforms $\phi$ of level $N$ and Laplacian eigenvalue $\lambda$. For notational convenience, we introduce the spectral parameter
\begin{equation}\label{SpectralParameter}
t := \sqrt{\lambda - 1}\in \RR \cup (-1, 1)i,
\end{equation}
and throughout this paper we write
\begin{equation}\label{Tdef}
T := \max(1,|t|) \asymp 1 + \sqrt{\lambda}.
\end{equation}

We shall see in Section~\ref{trivialboundsection} that fairly standard estimates yield the clean uniform bound
\begin{equation}\label{trivial}
 \| \phi \|_{\infty} \ll_{\eps} T (TV)^{\eps}.
\end{equation}
This statement is a quantitative analogue of the local geometric bound \eqref{UpperBound}
for compact Riemannian manifolds: we shall refer to it as the ``trivial bound", although its proof relies on various arithmetic
facts to the extent that its validity for general $N\in\OO$ is unclear\footnote{Very recently, a remarkable non-arithmetic approach for classical holomorphic cusp forms was given in \cite{FJK}.} (as opposed to squarefree $N\in\OO$ that we restrict to from now on).
We start with a result that improves \eqref{trivial} in the volume aspect.

\begin{theorem}\label{thm1} Let $N\in\OO$ be squarefree and $\phi$ an $L^2$-normalized Hecke--Maa{\ss} cuspidal newform on $\Gamma_0(N)\backslash \HH^3$ with spectral parameter $t$. Then for any $\eps > 0$ we have
\[\| \phi \|_{\infty} \ll_{\eps}  T V^{-1/6} (TV)^{\eps}.\]
\end{theorem}

This bound is of the same strength as the corresponding results for $\SL_2(\RR)$ \cite{HT2}, $\SO_3(\RR)$ \cite{BM}, $\SO_4(\RR)$ \cite{BM}, and probably as good as one can hope for with current technology. It is reasonable to expect in the volume aspect $\| \phi \|_{\infty} \ll_{T,\eps} V^{-1/2+\eps}$ (which would be optimal) for squarefree $N$, relative to which Theorem~\ref{thm1} is one-third along the way. In this sense, Theorem~\ref{thm1} matches the quality of classical Weyl-type subconvexity bounds, and is therefore a fairly natural result that seems hard to improve.

It is an equally interesting problem to establish a saving in the eigenvalue aspect. In this respect, we prove

\begin{theorem}\label{thm2} Let $N\in\OO$ be squarefree and $\phi$ an $L^2$-normalized Hecke--Maa{\ss} cuspidal newform on $\Gamma_0(N)\backslash \HH^3$ with spectral parameter $t$. Then for any $\eps > 0$ we have
\[\| \phi \|_{\infty} \ll_{\eps}  T^{5/6} (TV)^{\eps}. \]
\end{theorem}

The problem of obtaining a nontrivial bound in the situation of Theorem~\ref{thm2} was addressed by Koyama in the paper \cite{Ko}\footnote{Unfortunately, the argument in this work seems to have a gap, specifically the proof of \cite[Lemma~5.3]{Ko} is incomplete. Our counting argument is quite different from Koyama's method, though at one point we incorporate a crucial idea from his paper (see Section~\ref{TinyDistancesHighRange}).}. Our bound in Theorem~\ref{thm2} is of the same strength as the corresponding results for $\SL_2(\RR)$ \cite{IS} and $\SO_3(\RR)$ \cite{VdK}, even though the corresponding counting problems differ substantially; the fact that they all meet at $5/6$ of the trivial bound seems to indicate that this is the natural exponent in the eigenvalue aspect, and it is probably as good as one can hope for with current technology.
We also remark that the sup-norm estimate in Theorem~\ref{thm2} is uniform over the entire manifold $\Gamma_0(N)\backslash\HH^3$, including the cuspidal regions, where high-energy eigenstates  exhibit a sizable ``bump'' for purely analytic reasons~\cite{Sa}.

Theorems~\ref{thm1} and~\ref{thm2} can be combined to a \emph{hybrid} bound, which saves simultaneously in both aspects.

\begin{theorem}\label{thm3} Let $N\in\OO$ be squarefree and $\phi$ an $L^2$-normalized Hecke--Maa{\ss} cuspidal newform on $\Gamma_0(N)\backslash \HH^3$ with spectral parameter $t$. Then for any $\eps > 0$ we have
\[\| \phi \|_{\infty} \ll_{\eps} (TV^{1/2})^{-1/9} T (TV)^{\eps}.\]
\end{theorem}

\bigskip

We comment briefly on the methods employed and some of the principal features and difficulties encountered in the proofs of our main theorems. The general approach to sup-norms of automorphic forms that we employ is originally due to Iwaniec--Sarnak~\cite{IS}: it is based on an amplified pre-trace formula (Section~\ref{amplifiedsection} below) and is more or less identical in all treatments of this subject. Applying the amplified pre-trace formula with a suitable nonnegative test function and dropping all but one term leads to a diophantine problem on the geometric side. It is the treatment of this diophantine problem, which we present in Sections~\ref{DiophantineInequailitiesSection}--\ref{TinyDistancesSection}, that lies at the heart of the method, and this depends heavily on the considered manifolds. In the present setting, one has to count matrices $\gamma \in \M_2(\OO)$ (with determinant suitably bounded in terms of $T$ and $V$) that are very close to the stabilizer of the point $P \in \Gamma_0(N)\backslash \HH^3$ at which we want to bound $\phi$ (eigenvalue aspect) and have their lower left entry divisible by $N$ (volume aspect). The former condition leads to a complicated system of diophantine inequalities, while the latter condition can be exploited by methods from the geometry of numbers. A crucial role in our arguments in both aspects is played by our adaptation to the context of $\OO$-modules of Minkowski's lattice point counting argument in Lemma~\ref{lem3} and our ability to apply it in several different ways for various ranges of the parameters.

We mention some of the unique features and difficulties encountered when moving from $\SL_2(\RR)$ to $\SL_2(\CC)$. On a technical level, the action of $\GL_2(\CC)$ (which is the group where the Hecke operators live) on the upper half space $\HH^3$ is less convenient to work with than the action of $\GL_2(\RR)$ on the upper half plane $\HH^2$. This is reflected by the fact that arithmetic on $\HH^2$ (viewed as a subset of complex numbers) is commutative, while on $\HH^3$ (viewed as a subset of Hamiltonian quaternions) it is not; thus, compared to the commutative case, it is much harder to exploit diophantine conditions such as \eqref{BasicDiophantineConditions}--\eqref{third} in the context of an efficient lattice point count without fixing too many parameters. On a more conceptual level, it is clear from the above description that the sup-norm problem is highly sensitive to the structure of the stabilizer of a point. On $\HH^2$, the stabilizer is conjugate to $\SO_2(\RR)$, a commutative group homeomorphic to $S^1$. It is not hard to see (this being essentially a binary quadratic problem) that there cannot be too many lattice points on a (possibly deformed) circle. On $\HH^3$,  the stabilizer is conjugate to $\SU_2(\CC)$, a non-commutative group homeomorphic to $S^3$. It is much harder to see that on this bigger variety there are relatively few lattice points (this being essentially a quaternary quadratic problem). As a matter of fact, there are cases when this variety contains too many lattice points to allow an efficient matrix count, and (as remarked earlier) this feature is manifested by the power growth of eigenforms at distinguished points. We choose our amplifier in a way that avoids such situations.

In the eigenvalue aspect, where extremely delicate diophantine analysis is encountered in very small neighborhoods, some of the estimates are sensitive to the diophantine properties of $P = z + rj$, and we use different methods according to whether $z$ is ``well approximable" or ``badly approximable". This dichotomy is reminiscent of the method in \cite{BHo}, and we gain substantially by approximating genuinely in $K$.

Finally, we mention that this paper uses, for the first time, a slightly more optimized amplifier \eqref{ampli} that needs only second and fourth powers, but not first and third powers coming from mixed terms. This avoids some technical difficulties.

\begin{notation} Our two principal parameters are the spectral parameter $T$ and the volume parameter $V$. For convenience, we borrow from \cite{HT2} the notation
\begin{equation}\label{notationformula}
X \preccurlyeq Y \qquad \overset{\text{def}}\Longleftrightarrow \qquad X \ll_\eps Y(TV)^{\eps}
\end{equation}
for any two quantities $X$ and $Y$. This notation will be in force for the rest of the paper.
\end{notation}

\begin{acknowledgement} We thank Peter Sarnak for valuable discussions on the presentation in this paper and the referees for a careful reading of the manuscript.
\end{acknowledgement}

\section{Groups and matrices}

Recall the description of the upper half space $\HH^3$ in terms of the Hamiltonian quaternions as above \eqref{parts}. The group $\SL_2(\CC)$ acts on $\HH^3$ by
\begin{equation}\label{action}
g P = (aP+b)(cP+d)^{-1}, \qquad g =\begin{pmatrix} a & b\\c & d\end{pmatrix} \in \SL_2(\CC),
\end{equation}
where the inverse and the multiplication are performed in the quaternion division algebra. This action factors through $\text{Isom}^{+}(\HH^3)\cong\PSL_2(\mathbb{C})\cong\PGL_2(\mathbb{C})$ and in turn lifts uniquely to an action of $\GL_2(\CC)$ on $\HH^3$ determined by \eqref{action} and the fact that its center acts trivially:
\begin{equation}\label{action3}
\begin{pmatrix} a & \\ & a\end{pmatrix}P = P, \qquad a \in \CC^{\times}.
\end{equation}
We note that the formula in \eqref{action} is valid when $\det g>0$, but fails otherwise. At any rate, we have the general relation (cf.\ \eqref{parts})
\begin{equation}\label{ratio}
\frac{\Im (g P)}{\Im(P)} = \frac{|\det g|}{\|cP+d\|^2}=\frac{ |\det g|}{|cz + d|^2 + |cr|^2},\qquad g \in \GL_2(\CC),
\end{equation}
where $\| P \| := (P\ov{P})^{1/2}$ is the usual quaternionic norm, and the absolute value is the usual absolute value of complex numbers (not to be confused with the number field norm of $K$). See \cite{EGM} for more details.

For $n \in \OO\setminus\{0\}$ we write
\[R(n)  := \left\{\begin{pmatrix} a & b\\c & d\end{pmatrix} \in \M_2(\OO) : N \mid c,\ (a, N) = 1,\ ad-bc = n\right\}.\]
In particular, $R(1) = \Gamma_0(N)$. We define the Hecke operator $T_n$ on functions $\phi:\Gamma_0(N)\backslash\HH^3\to\CC$ by
\begin{equation}\label{hop}
(T_n \phi)(P ) := \frac{1}{|n|} \sum_{g \in R(1) \backslash R(n)} \phi(gP)
= \frac{1}{4|n|}\sum_{\substack{ad=n\\ (a, N) = 1}}\,\sum_{b\mod d}\phi\left(\begin{pmatrix}a & b\\0 & d\end{pmatrix}P\right).
\end{equation}
In this normalization, the Ramanujan conjecture asserts that the spectrum of $T_n$ acting on cusp forms is bounded by $O_\eps(|n|^{\eps})$.
Note that the matrices $\left(\begin{smallmatrix}a & b\\0 & d\end{smallmatrix}\right)$ on the right hand side represent each coset in $R(1) \backslash R(n)$ exactly $4$ times, as follows by adapting the proof of \cite[Prop.~3.36]{Shi}. Note also that for $(n,N)=1$ the condition $(a,N)=1$ is automatic. The Hecke operators for $n$ not coprime with $N$ are only used in Section~\ref{trivialboundsection}.

It is clear from the definition that $T_i$ is the involution induced by the rotation by $\pi/2$ around the $r$-axis, and $T_iT_n=T_{in}$ for any $n \in \OO\setminus\{0\}$. In particular, $T_{-n}=T_n$. The Hecke operators commute with each other and the Laplace operator, and they satisfy the multiplicativity relation
\begin{equation}\label{hecke}
T_mT_n = \sum_{\substack{(d)|(m,n)\\(d,N)=1}} T_{mn/d^2},\qquad m,n\in \OO\setminus\{0\},
\end{equation}
where the sum is over the ideals in $\ZZ[i]$ coprime with $N$ that divide $m$ and $n$. This relation follows by adapting the proof of \cite[Theorem~6.6]{Iw}, but see also \cite[Corollary~4.3]{Zh} for the analogous result over ideals. Finally, we remark that
the Hecke operators $T_n$ for $(n, N) = 1$ are self-adjoint on $L^2(\Gamma_0(N)\backslash\HH^3)$, as follows by adapting the proof of
\cite[Theorem~6.20]{Iw}.

Following \cite{HT1}, we consider (for $N\in\OO$ squarefree)
\begin{equation}\label{normalizer}
\Gamma_0^{\ast}(N):=\left\{\begin{pmatrix} Ma & b\\Mc & Md\end{pmatrix} \in \PGL_2(\CC):
a, b, c, d \in \OO,\ M\mid N\mid Mc,\ Mad-bc=1\right\};
\end{equation}
see also \cite[Section~5.3]{Li}. This set is a discrete subgroup of $\PGL_2(\CC)$ containing (the image of) $\Gamma_0(N)$ as a normal subgroup.
The quotient group $\Gamma_0^{\ast}(N)/\Gamma_0(N)$ is isomorphic to $(\ZZ/2\ZZ)^{\omega(N)}$, where $\omega(N)$ is the number of prime ideal factors
of $N$. The action of $\Gamma_0^{\ast}(N)$ commutes with the Hecke operators $T_n$ for $(n,N)=1$, hence an element of $\Gamma_0^{\ast}(N)$ acts by $\pm 1$ on every newform $\phi$ for $\Gamma_0(N)$, by multiplicity one. For the purpose of estimating $|\phi(P)|$ we may therefore restrict the point $P\in\HH^3$ to any fundamental domain of $\Gamma_0^{\ast}(N) \backslash \HH^3$. We shall work with the analogue of the Ford polygon (see \cite[Section~2.3]{Iw}):
\begin{equation}\label{fund}
\FF(N):=\left\{P=z+rj\in\HH^3:\text{$\Im(P)\geq\Im(gP)$ for all $g \in \Gamma^{\ast}_0(N)$},\
|\re(z)|\leq 1/2,\ 0\leq \im(z) \leq 1/2\right\}.
\end{equation}
The next section shows that this fundamental domain has very convenient properties.

\section{Geometry of numbers}

We start our discussion with a result in the geometry of numbers that goes back to Minkowski. For any lattice $\Lambda\subset\RR^n$, it is easy to see that the lattice minimum
\[m_1:=\min\left\{\|x-y\|:\text{$x,y\in\Lambda$ and $x\neq y$}\right\}\]
is positive and that (by a volume argument) the number of lattice points inside any ball of radius $R>0$ is $\ll_n 1+(R/m_1)^n$. The following result is an essentially sharp refinement of this count that takes into account the finer shape of $\Lambda$. It is a kind of Lipschitz principle, also implicit in the corresponding diophantine analysis for the sup-norm problem on $\SO_3(\RR)$ \cite[Section~4]{BM}.

\begin{lemma}\label{lem2} Let $\Lambda \subset \RR^n$ be a lattice with successive minima $m_1 \leq m_2 \leq \dots \leq m_n$. Let $B\subset\RR^n$ be a ball of radius $R>0$ and arbitrary center. Then
\[\#(\Lambda \cap B) \ll_n 1 + \frac{R}{m_1} + \frac{R^2}{m_1m_2} + \dots + \frac{R^n}{m_1 m_2\cdots m_n}.\]
\end{lemma}

\begin{proof} This is essentially \cite[Lemma~2]{Sch} or \cite[Prop.~2.1]{BHW}, but for sake of completeness we provide the proof. If $\#(\Lambda \cap B)\leq 1$, then we are done. Otherwise, let us fix a point $\tilde u\in \Lambda \cap B$, then any other point $u\in \Lambda \cap B$ is determined by the difference $u':=u-\tilde u\in\Lambda$. The ball $B$ has diameter $2R$, hence $u'\in\Lambda\cap B'$, where $B'$ is the ball of radius $2R$ centered at the origin. By assumption, the subspace $V\subset\RR^n$ generated by $\Lambda \cap B'$ has positive dimension $1\leq k\leq n$, and $\Lambda\cap V$ is a $k$-dimensional lattice in $V$ with a fundamental parallelepiped lying in the $k$-ball $kB'\cap V$. The fundamental parallelepiped has $k$-volume $\asymp_k m_1m_2\cdots m_k$ by a theorem of Minkowski~\cite[Theorem 3 on p.~124]{GL}, and its translates by the elements of $\Lambda\cap B'$ are pairwise disjoint and lie in the $k$-ball $(k+1)B'\cap V$ whose $k$-volume is $\asymp_k R^k$. Therefore
\[\#(\Lambda \cap B) \leq \#(\Lambda \cap B') \ll_k R^k/(m_1 m_2\cdots m_k),\]
and we are done again.
\end{proof}

We identify the $4$-dimensional $\RR$-vector space of Hamiltonian quaternions with $\RR^4$ via the standard basis $\{1,i,j,k\}$; then, the quaternion norm agrees with the Euclidean norm in $\RR^4$. For an arbitrary point $P\in\FF(N)$ lying in the fundamental domain given by \eqref{fund}, we consider the $\ZZ$-lattice
\[\Lambda(P) := \{cP+d: c,d \in \OO\}. \]
The next lemma shows how Lemma~\ref{lem2} and the restriction to $P\in\FF(N)$ lead to a very efficient count of lattice points in $\Lambda(P)$ that will be of great use for us in various diophantine situations.

\begin{lemma}\label{lem3} Let $P = z + rj \in \FF(N)$ for $N$ squarefree. Then the lattice $\Lambda(P)$ and its successive minima
$m_1\leq m_2\leq m_3\leq m_4$ satisfy:
\begin{itemize}
\item[a)] $m_1m_2m_3m_4 \asymp r^2$;
\item[b)] $m_1 \geq |N|^{-1/2}$ and $r\gg |N|^{-1}$;
\item[c)] $m_1=m_2$ and $m_3=m_4$;
\item[d)] in any ball of radius $R$ the number of lattice points is $\ll 1+R^2|N|+R^4r^{-2}$.
\end{itemize}
\end{lemma}

\begin{proof} a) Calculating the exterior product of the $\ZZ$-basis $\{1,i,P,iP\}$ shows that $\Lambda(P)$ has covolume $r^2$.
Hence $m_1m_2m_3m_4 \asymp r^2$ follows from Minkowski's theorem~\cite[Theorem 3 on p.~124]{GL}.

b) We follow the proof of \cite[Lemma~2.2]{HT1}. The bound $m_1 \geq |N|^{-1/2}$ is equivalent to the statement that $\|cP + d\|^2\geq |N|^{-1}$ for any $(c,d)\in\OO^2$ distinct from $(0,0)$. Without loss of generality, we may assume that $c$ and $d$ are coprime. As $N$ is squarefree, $M := N/(N, c)$ is coprime to $c$, hence there exist $a,b\in\OO$ such that $Mad-bc = 1$. The matrix $g = \left(\begin{smallmatrix} Ma & b\\ Mc & Md\end{smallmatrix}\right)$ lies in
$\Gamma_0^{\ast}(N)$, hence by \eqref{ratio} and \eqref{fund},
\[1 \geq \frac{\Im(gP)}{\Im(P )} = \frac{|M|}{\| McP + Md\|^2} \geq \frac{|N|^{-1}}{\|cP + d\|^2}.\]
Therefore, $m_1 \geq |N|^{-1/2}$. By part a) we infer that $r^2\gg |N|^{-2}$ and so $r\gg |N|^{-1}$.

c) By the definition of $m_3$, there are three $\ZZ$-independent vectors $u_1, u_2, u_3 \in \Lambda(P)$ with norms at most $m_3$. Write $V := \QQ u_1 + \QQ u_2 + \QQ u_3$. Then $V$ is a $3$-dimensional $\QQ$-vector space, but not a $K$-vector space (since $3$ is odd). Therefore, $Ku_1 + Ku_2 + Ku_3$ is strictly larger than $V$. Since $K=\QQ+\QQ i$, at least one element from $\{iu_1,iu_2,iu_3\}$ lies outside $V$. Adding this element to $\{u_1,u_2,u_3\}$, we obtain four $\ZZ$-independent vectors in $\Lambda(P)$ with norms at most $m_3$. This shows that $m_3\leq m_4\leq m_3$. The proof of $m_1=m_2$ is similar.

d) By Lemma~\ref{lem2}, the number of lattice points in any ball of radius $R$ is
\[\ll 1 + \frac{R}{m_1} + \frac{R^2}{m_1m_2} + \frac{R^3}{m_1m_2m_3} + \frac{R^4}{m_1m_2m_3m_4}.\]
By part c), we can drop the second and fourth term, because they are the geometric means of their neighbors. Estimating the remaining terms using parts a) and b), the above expression is $\ll 1+R^2|N|+R^4r^{-2}$, as stated.
\end{proof}

\section{Fourier expansion and trivial bound}\label{trivialboundsection}

In this section, we establish a ``trivial bound" for an $L^2$-normalized Hecke--Maa{\ss} cuspidal newform using Rankin--Selberg theory and the Fourier expansion.

An $L^2$-normalized cusp form $\phi$ on $\Gamma_0(N) \backslash \HH^3$ with Laplacian eigenvalue $\lambda = 1 + t^2$  has a Fourier expansion
\[\phi(z + rj ) = r\sum_{n \in \OO\setminus\{0\}} \rho(n) K_{it}(2 \pi |n| r)e(\re(nz)),\]
where $\rho(-n)=\rho(n)$ holds by the invariance of $\phi$ under
$\left(\begin{smallmatrix}-1 & 0\\0 & 1\end{smallmatrix}\right)=\left(\begin{smallmatrix}i & 0\\0 & i\end{smallmatrix}\right)\left(\begin{smallmatrix}i & 0\\0 & -i\end{smallmatrix}\right)$.
This Fourier expansion is based on the nondegenerate $\RR$-bilinear form $(w,z)\mapsto\re(wz)$ as opposed to the standard scalar product $(w,z)\mapsto\langle w,z\rangle$ of $\CC=\RR^2$ used by \cite[Section~3.3]{EGM}. That is, we specialize to our setting the general Fourier--Whittaker expansion from the theory of automorphic forms on $\GL_2$, in order to be compatible with the Hecke operators (cf.\ \cite{Ma,Sz,Zh}). Indeed, adapting the proof of \cite[Prop.~6.3]{Iw}, it follows now for any $n \in \OO\setminus\{0\}$ that
\[(T_n\phi)(z+rj)=r\sum_{m \in \OO\setminus\{0\}} \Biggl(\sum_{\substack{(d)|(m,n)\\(d,N)=1}} \rho\left(\frac{mn}{d^2}\right)\Biggr) K_{it}(2 \pi |m| r)e(\re(mz)).\]
This formula is closely related to the multiplicativity relation \eqref{hecke}.

Let us also assume that $\phi$ is a newform, then the previous identity implies
\[\rho(n)=\rho(1)\lambda(n),\qquad n \in \OO\setminus\{0\},\]
where $\lambda(n)$ is the $n$-th Hecke eigenvalue of $\phi$. By Rankin--Selberg theory combined with (an extension of) a famous bound of Hoffstein--Lockhart~\cite{HL} and a famous trick of Iwaniec~\cite[(19)]{Iw2} one shows that (cf.\ \cite[Section~2]{Ko}, \cite[Section~6.1]{EGM}, \cite[Prop.~3.2]{Ma})
\[|\rho(1)|^2\preccurlyeq|\rho(1)|^2\,\res_{s=1}L(s,\phi\otimes \tilde\phi)
\preccurlyeq\res_{s=1}\sum_{n \in \OO\setminus\{0\}}\frac{|\rho(n)|^2}{|n|^{2s}}
\ll\frac{1}{|\Gamma(1 + it)|^2 V}\ll\cosh(\pi t)\frac{1}{TV},\]
and also that
\[\sum_{1\leq |n| \leq x} |\lambda(n)|^2 \preccurlyeq x^2,\qquad x\in\RR.\]
Here we used the notations \eqref{Tdef} and \eqref{notationformula}. These relations imply
\begin{equation}\label{RS}
\sum_{1\leq |n| \leq x} |\rho(n)|^2 =|\rho(1)|^2\sum_{1\leq |n| \leq x} |\lambda(n)|^2\preccurlyeq \cosh(\pi t)\frac{x^2}{TV},\qquad x\in\RR.
\end{equation}
From this bound we shall derive\cprotect\footnote{See \verb|http://www2.toyo.ac.jp/~koyama| for a corrected version of Koyama's corresponding Proposition~3.1.}

\begin{lemma}\label{fourier} Let $N\in\ZZ[i]$ be nonzero and $\phi$ an $L^2$-normalized Hecke--Maa{\ss} cuspidal newform on $\Gamma_0(N)\backslash \HH^3$ with spectral parameter $t \in \RR \cup (-1, 1)i$. Then for $P = z + rj \in \HH^3$ we have
\begin{equation}\label{finalphiPbound}
\phi(P) \preccurlyeq  \frac{T}{V^{1/2} r} + \frac{T^{1/2}}{V^{1/2}r^{1/3}}.
\end{equation}
\end{lemma}

\begin{proof} For $t\in\RR$ we can assume $t\geq 0$ without loss of generality. For $n \in \OO\setminus\{0\}$ we recall the bound
\[\cosh(\pi t/2) K_{it}(2\pi |n| r) \ll_\eps
\begin{cases}
(|n|r)^{-|\im(t)| - \eps}, & T=1,\ |n|\leq 1/r;\\
\min\left(t^{-1/3}, |(2\pi |n| r)^2 - t^2|^{-1/4}\right), & T > 1;\\
e^{-|n|r}, & |n| > T/r;
\end{cases}\]
see e.g.\ \cite[p.~679]{BHo} and \cite[Prop.~9]{HM}. In particular, for $r>T$,
the third case applies to every $n \in \OO\setminus\{0\}$, and we conclude
\[|\phi(P )| \leq r \sum_{k=0}^{\infty}  \sum_{2^k \leq |n| < 2^{k+1}} |\rho(n) K_{it}(2 \pi |n| r)|
\preccurlyeq r \sum_{k=0}^{\infty}  \left(\frac{4^{k+1}}{TV}\right)^{1/2} \left( 4^{k+1} e^{-2^{k+1}r} \right)^{1/2}
\ll \frac{re^{-r}}{(TV)^{1/2}}.\]
Hence we can assume that
$r\leq T$. Then, by a similar argument, the tail $|n|>(TV)^{\eps} T/r$ contributes $O_\eps((TV)^{-10})$, which
is admissible.

Let us focus on the case $T>1$ (i.e.\ $t>1$). By Cauchy--Schwarz and \eqref{RS} we conclude
\begin{equation}\label{phiPbound}
\phi(P) \preccurlyeq r\frac{T/r}{(TV)^{1/2}}
\left(\sum_{1\leq |n|\leq(TV)^{\eps} T/r} \min\left(t^{-2/3}, |(2\pi |n| r)^2 - t^2|^{-1/2}\right)\right)^{1/2}+(TV)^{-10}.
\end{equation}
Denoting
\[A(u,v):=\#\{n \in \OO\setminus\{0\} : ut^{4/3}\leq|(2\pi |n| r)^2 - t^2|<vt^{4/3}\},\qquad 0\leq u<v\leq R:=50(TV)^{2\eps}T^{2/3},\]
we clearly have
\[\sum_{1\leq |n|\leq(TV)^{\eps} T/r} \min\left(t^{-2/3}, |(2\pi |n| r)^2 - t^2|^{-1/2}\right)\leq
T^{-2/3}\left(A(0,1)+\sum_{0\leq k\leq \log_2 R} 2^{-k/2}A(2^k,2^{k+1})\right).\]
Using Sierpi\'nski's classical bound for the circle problem,
\[\#\{ n \in \ZZ[i] : 1\leq |n| \leq X\} = \pi X^2 + O(X^{2/3}),\]
we infer
\[A(u,v)\preccurlyeq\frac{(v-u)T^{4/3}}{r^2}+\frac{T^{2/3}}{r^{2/3}},\]
so that
\[\sum_{1\leq |n|\leq(TV)^{\eps} T/r} \min\left(t^{-2/3}, |(2\pi |n| r)^2 - t^2|^{-1/2}\right)\preccurlyeq  \frac{T}{r^2}+\frac{1}{r^{2/3}} .\]
Substituting this into \eqref{phiPbound}, we obtain \eqref{finalphiPbound} in the case $T>1$.

For $T=1$ (i.e. $|t| \leq 1$), and for $0<\eps<1-|\im(t)|$ as we can assume without loss of generality, we have by a similar but simpler argument
\[\phi(P) \preccurlyeq \frac{1}{V^{1/2}}
\left(\sum_{1\leq |n|\leq(TV)^{\eps}/r} (|n|r)^{-2|\im(t)| - 2\eps} \right)^{1/2} +V^{-10} \preccurlyeq \frac{1}{V^{1/2} r}.\]
The proof of Lemma~\ref{fourier} is complete.
\end{proof}

Part~b) of Lemma~\ref{lem3} shows that $|\phi(P)|$ attains its maximum at a point $P=z+rj\in\FF(N)$ with $r\gg_\eps V^{-1/2-\eps}$ (when $N\in\ZZ[i]$ is squarefree), hence \eqref{finalphiPbound} justifies our earlier claim \eqref{trivial}.

\section{Amplified pre-trace formula}\label{amplifiedsection}

In this section, we construct an identity that allows us to reduce our sup-norm problem to a $4$-dimensional counting problem over the Gaussian integers. The \emph{amplified pre-trace formula} to be described below is based on the work of Selberg~\cite{Sel} and Duke--Friedlander--Iwaniec~\cite{DFI}. The pre-trace formula itself is an explicit spectral decomposition of the automorphic kernel, connecting a spectral sum with a geometric sum, and it also serves as the starting point for the fundamental Selberg trace formula in the theory of automorphic forms.

For the amplified pre-trace formula we need a basis of the Hilbert space
\[L^2(\Gamma_0(N)\backslash\HH^3)=\CC\oplus L^2_\text{cusp}(\Gamma_0(N)\backslash\HH^3)\oplus L^2_\text{cont}(\Gamma_0(N)\backslash\HH^3)\]
consisting of $\Gamma_0(N)$-invariant functions on $\HH^3$ that are simultaneous eigenfunctions of the Laplace operator and the Hecke
operators $T_n$ for $(n,N)=1$ (cf.\ \eqref{hop}). The spectral theory of automorphic forms combined with Hecke theory readily provides us with such an orthonormal basis $\{\phi_j \}_{j=1}^\infty$ for the cuspidal subspace, and we can assume that it contains the newform $\phi=\phi_{j_0}$
whose sup-norm we want to bound in Theorems~\ref{thm1} to~\ref{thm3}. We complement this basis with the constant function $\phi_0:=\vol(\Gamma_0(N)\backslash\HH^3)^{-1/2}$. For $j\geq 0$, we write $t_j \in \RR \cup [-1, 1]i$ for the spectral parameter of $\phi_j$ as in \eqref{SpectralParameter}, and we denote the Hecke eigenvalues of $\phi_j$ by $\lambda_j(n)\in\RR$ for $(n,N)=1$.

One approach to describing $L^2_{\text{cont}}(\Gamma_0(N)\backslash\HH^3)$ involves Eisenstein series $E_{\mathfrak{a}}(P,s)$ attached to the various cusps $\mathfrak{a}$; however, such an Eisenstein series is in general not an eigenfunction of the Hecke operators. Hecke eigenforms for the continuous spectrum arise naturally in an adelic treatment, see \cite{GJ} for precise definitions and details. In this language, Eisenstein series are associated to the primitive Gr\"ossencharacters $\chi$ of $K$ with trivial infinite part whose conductor squared divides $N$ and, for each such $\chi$, to the elements $\varphi$ of a finite orthonormal basis of an induced representation determined by $\chi$ (cf.\ \cite[Section~2.4]{Ma}). We denote the corresponding Eisenstein series by $E_{\chi,\varphi,s}(P)$, where $s\in\CC$ and $P\in\HH^3$, and record its Hecke eigenvalues
\begin{equation}\label{lambdachi}
\lambda_{\chi,s}(n)=\frac{1}{4}\sum_{ad=n}\frac{\chi_\text{fin}(a)|a|^s}{\chi_\text{fin}(d)|d|^s},\qquad (n,N)=1.
\end{equation}
In particular, the Hecke eigenvalues are independent of the parameter $\varphi$, and they are real for $s\in i\RR$ since in this case complex conjugating the above expression has the same effect as switching $a$ and $d$.

\begin{remark} If $N\in\OO$ is squarefree, which is the main focus of this paper, then $\chi$ is trivial, and the Hecke eigenvalues $\lambda_{1,s}(n)$ above are those of the classical Eisenstein series $E(P,1+s)$ at the cusp $\infty$. In fact, in this case, $E_{1,\varphi,s}(P)$ for suitable parameters $\varphi$ agree with $E_\mathfrak{a}(P,1+s)$ at the various cusps $\mathfrak{a}$, as the latter are linear combinations of $E(dP,1+s)$ for $d\mid N$. See \cite[p.~1187]{CI} for a direct argument over $\QQ$, and \cite[Section~2.1]{Sz} for a discussion of the case $N=1$.
\end{remark}

Using the above Hilbert space basis, the spectral expansion of automorphic kernels reads (cf.\ \cite[Section~3.5]{EGM}, \cite[Sections~2.4--2.5]{Sz})
\begin{equation}\label{pretrace}
\sum_{j=0}^\infty \phi_j(P_1 ) \ov{\phi_j(P_2)}h(t_j) +
\sum_{(\chi, \varphi)}\frac{1}{4\pi}\int_{-\infty}^{\infty} E_{\chi,\varphi,it}(P_1)\ov{E_{\chi,\varphi,it}(P_2)}h(t)dt \\
= \sum_{\gamma \in \Gamma_0(N)} k(u(\gamma P_1, P_2))
\end{equation}
for any two points $P_1, P_2 \in \HH^3$, and any even, holomorphic, rapidly decaying function $h(t)$ defined on the strip $|\im(t)| < 1+\eps$
with Selberg/Harish-Chandra transform $k(u)$. Here $u(P_1, P_2)$ is the basic point-pair invariant (cf.\ \eqref{parts})
\begin{equation}\label{udef}
u(P_1, P_2) :=  \frac{\| P_1 - P_2\|^2}{2\Im(P_1)\Im(P_2)} = \cosh(\text{dist}(P_1, P_2)) - 1,
\end{equation}
and the test functions $h(t)$ and $k(u)$ are connected by
\begin{equation}\label{testfunctions}
h(t)= \int_{-\infty}^\infty g(x)e^{itx}dx,\qquad k(u)=\frac{-g'(v)}{2\pi\sqrt{u^2+2u}},\qquad v:=\log(1+u+\sqrt{u^2+2u}).
\end{equation}

For a given $P\in\HH^3$ and $n\in\OO$ with $(n,N)=1$, we apply \eqref{pretrace} for all pairs $(P_1,P_2)=(gP,P)$, where $g$ runs
through a set of representatives for the coset space $R(1) \backslash R(n)$. The sum of the resulting identities gives, by \eqref{hop} and the choice of $\phi_j$ and $E_{\chi,\varphi,it}$ as a Hecke eigenbasis,
\begin{equation}\label{pretrace1}
\sum_{j=0}^\infty \lambda_j(n)|\phi_j(P )|^2 h(t_j) + \sum_{(\chi, \varphi)} \frac{1}{4\pi} \int_{-\infty}^{\infty} \lambda_{\chi, it}(n) |E_{\chi, \varphi, it}(P) |^2 h(t) dt = \frac{1}{|n|} \sum_{\gamma \in R(n)} k(u(\gamma P, P)).
\end{equation}
We stress that all the Hecke eigenvalues $\lambda_j(n)$ and $\lambda_{\chi,it}(n)$ here are real by the self-adjointness of $T_n$ for $(n,N)=1$ and the remark below \eqref{lambdachi}, and they satisfy a multiplicativity relation coming from \eqref{hecke}. Next, we linearly combine these eigenvalues for various values of $n$ to an amplifier $A_j\geq 0$ and $A_{\chi,it}\geq 0$ that emphasizes our preferred form $\phi=\phi_{j_0}$.

Let $(TV)^{\eps}\leq L\leq TV$ be a parameter to be optimized later, and let us assume that $TV$ exceeds a large constant depending only on $\eps$. Then, in particular, $L$ is sufficiently large in terms of $\eps$. Set
\[x(l) := \sgn(\lambda_{j_0}(l)),\qquad (l,N)=1,\]
\begin{equation}\label{primeset}
P(L) := \{l \in \ZZ[i] \text{ prime}: l\nmid 2N,\ 0 < \arg(l) < \pi/4,\ L \leq |l|^2 \leq 2L\}.
\end{equation}
Then $P(L)$ consists of $\asymp L/\log L$ split primes, no two of which are associated.
We shall also need the fact that $l_1l_2$ does not have a nontrivial divisor $m \in \ZZ$ when $l_1, l_2 \in P(L)$. We consider the amplifier, inspired by \cite[(4.11)]{Ve},
\begin{equation}\label{ampli}
\begin{split}
A_{j} :&= \left(\sum_{l \in P(L)} x(l) \lambda_j( l)\right)^2 + \left(\sum_{l \in P(L)} x(l^2) \lambda_j( l^2)\right)^2 \\
& = \sum_{l \in P(L)}y_0(l)+ \sum_{l_1, l_2 \in P(L)} y_1(l_1l_2) \lambda_j(l_1l_2) + \sum_{l_1, l_2 \in P(L)}y_2(l_1^2l_2^2) \lambda_j(l_1^2l_2^2),
\end{split}
\end{equation}
along with a similarly defined expression for $A_{\chi,it}$, where (cf.\ \eqref{hecke})
\begin{equation}\label{boundy}
\begin{split}
y_0(l) &:= x(l)^2 + x(l^2)^2 \ll  1,\\
y_1(l_1l_2) &:= x(l_1)x(l_2)+ \delta_{l_1=l_2}x(l_1^2)x(l_2^2) \ll 1, \\
y_2(l_1^2l_2^2) &:= x(l_1^2)x(l_2^2) \ll 1.
\end{split}
\end{equation}
The multiplicativity relation \eqref{hecke} implies that $|\lambda_{j_0}(l)|+|\lambda_{j_0}(l^2)|\gg 1$, and so
\begin{equation}\label{lower}
A_{j_0} = \left(\sum_{l \in P(L)} |\lambda_{j_0}( l)|\right)^2 + \left(\sum_{l \in P(L)} | \lambda_{j_0}( l^2)|\right)^2 \geq \frac{1}{2}\left(\sum_{l \in P(L)} |\lambda_{j_0}( l)| + |\lambda_{j_0}(l^2)|\right)^2     \gg_\eps L^{2-\eps}.
\end{equation}
Applying \eqref{pretrace1} for $n=1$, $n=l_1l_2$, $n=l_1^2l_2^2$, where $l_1,l_2\in P(L)$, yields by \eqref{ampli}
\begin{equation}\label{pretrace3}
\begin{split}
&\sum_{j=0}^\infty A_j |\phi_j(P )|^2 h(t_j) +
\sum_{(\chi, \varphi)} \frac{1}{4\pi} \int_{-\infty}^{\infty}  A_{\chi, it} |E_{\chi, \varphi, it}(P) |^2 h(t) dt = \sum_{l \in P(L)} y_0(l)\sum_{\gamma \in R(1)} k(u(\gamma P, P))\\
&+ \sum_{l_1, l_2 \in P(L)} \frac{y_1(l_1l_2)}{|l_1l_2|}\sum_{\gamma \in R(l_1l_2)} k(u(\gamma P, P))
+ \sum_{l_1, l_2 \in P(L)} \frac{y_2(l_1^2l_2^2)}{|l_1^2l_2^2|}\sum_{\gamma \in R(l_1^2l_2^2)} k(u(\gamma P, P)).
\end{split}
\end{equation}
This identity is what we call the amplified pre-trace formula.

We shall use \eqref{pretrace3} with a specific choice of the Selberg/Harish-Chandra transform pair $(h,k)$, which will be constructed so as to emphasize the contribution of the terms neighboring $|\phi_{j_0}(P)|^2$ (and keep all other terms non-negative) on the left-hand side
while forcing the right hand side to decay as quickly as possible. Inspired by \cite[(1.5)]{IS}, we consider $h(t)$ and $k(u)$ determined by \eqref{testfunctions} for
\begin{equation}\label{gformula}
g(x):=\frac{A\cos(Tx)}{2\pi \cosh(Ax)},\qquad T := \max(1,|t_{j_0}|),
\end{equation}
where $A>2$ is a large constant. The next lemma summarizes the properties that we need of these functions.
\begin{lemma}\label{lem4} The function $h(t)$ is even, holomorphic, and rapidly decaying in the strip $|\im(t)|<A$. It satisfies
the bound
\begin{equation}\label{boundh}
h(t)>\begin{cases}
0,&t\in\RR\cup (-A,A)i;\\
\frac{1}{8},&t\in\RR\cup (-\frac{A}{2},\frac{A}{2})i\text{ and }|t|\in(T-\frac{A}{2},T+\frac{A}{2}).
\end{cases}
\end{equation}
The Selberg/Harish-Chandra transform $k(u)$ of $h(t)$ satisfies the bound
\begin{equation}\label{boundk}
k(u) \ll_A \min\left(T^2, \frac{T}{\sqrt{u}(1+u)^A}\right),\qquad u\geq 0.
\end{equation}
\end{lemma}

\begin{proof} The function $g(x)$ defined by \eqref{gformula} is even and satisfies the bound $g(x)\ll_A e^{-A|x|}$, hence its Fourier transform $h(t)$ in \eqref{testfunctions} is even and holomorphic in the strip $|\im(t)|<A$ by Morera's theorem. Using the formula for the Fourier transform of $\sech x=(\cosh x)^{-1}$, see \cite[p.~81]{SS} or \cite[17.34.30]{GR}, we calculate
\begin{align}
\label{hformula2}h(t)&=\frac{1}{4}\sech\left(\frac{\pi(t+T)}{2A}\right) + \frac{1}{4}\sech\left(\frac{\pi(t-T)}{2A}\right)\\[4pt]
\label{hformula}&=\frac{\cosh\left(\frac{\pi t}{2A}\right)\cosh\left(\frac{\pi T}{2A}\right)}
{\cosh\left(\frac{\pi t}{A}\right)+\cosh\left(\frac{\pi T}{A}\right)},
\end{align}
which shows that $h(t)$ is rapidly decaying in the strip $|\im(t)|<A$.

We prove \eqref{boundh}. Assume first that $t\in\RR$. From \eqref{hformula2} it is clear that $h(t)>0$, because both terms are positive. Moreover, if $|t|\in(T-\frac{A}{2},T+\frac{A}{2})$ then one of these terms exceeds $\frac{1}{4}\sech\left(\frac{\pi}{4}\right)>\frac{1}{8}$, so that $h(t)>\frac{1}{8}$. Assume now that $t\in(-A,A)i$. From \eqref{hformula} it is clear that $h(t)>0$, because $\cosh\left(\frac{\pi t}{2A}\right)>0$ and $\cosh\left(\frac{\pi t}{A}\right)>-1$. Moreover, if $t\in(-\frac{A}{2},\frac{A}{2})i$ then $\cosh\left(\frac{\pi t}{2A}\right)>\cos\left(\frac{\pi}{4}\right)$ and $\cosh\left(\frac{\pi t}{A}\right)\leq 1$, so that
\[h(t)>\frac{\cos\left(\frac{\pi}{4}\right)\cosh\left(\frac{\pi T}{2A}\right)}{1+\cosh\left(\frac{\pi T}{A}\right)}
=\frac{\cos\left(\frac{\pi}{4}\right)}{2\cosh\left(\frac{\pi T}{2A}\right)}.\]
If, in addition, $|t|\in(T-\frac{A}{2},T+\frac{A}{2})$, then $T=T-|t|+ |t| < A/2 + A/2 = A$, whence
\[h(t)>\frac{\cos\left(\frac{\pi}{4}\right)}{2\cosh\left(\frac{\pi}{2}\right)}>\frac{1}{8}.\]

We prove \eqref{boundk}. By \eqref{gformula} and \eqref{testfunctions} we calculate
\[k(u)=\frac{AT\sin(Tv)+A^2\cos(Tv)\tanh(Av)}{(2\pi)^2\sqrt{u^2+2u}\cdot\cosh(Av)},\qquad u\geq 0.\]
It is straightforward to see that the numerator
\[AT\sin(Tv)+A^2\cos(Tv)\tanh(Av)\ll_A\min(T^2v,T)\ll\min(\sqrt{u^2+2u}\cdot T^2,T),\]
while the denominator
\[(2\pi)^2\sqrt{u^2+2u}\cdot\cosh(Av)\gg \sqrt{u^2+2u}\cdot(1+u)^A.\]
These two bounds clearly imply \eqref{boundk}, because $\sqrt{u^2+2u}\geq\sqrt{u}$.
\end{proof}

Combining \eqref{pretrace3} with the bounds \eqref{boundy}, \eqref{lower}, \eqref{boundh}, we obtain
\begin{equation}\label{pretrace2}
\begin{split}
&L^{2-\eps} | \phi_{j_0}(P ) |^2 \ll_\eps
\ L \sum_{\gamma \in R(1)} |k(u(\gamma P, P))|\\
&+ \sum_{l_1, l_2 \in P(L)} \sum_{\gamma \in R(l_1l_2)} \frac{|k(u(\gamma P, P))|}{|l_1l_2|}
+ \sum_{l_1, l_2 \in P(L)} \sum_{\gamma \in R(l_1^2l_2^2)} \frac{|k(u(\gamma P, P))|}{|l_1^2l_2^2|}.
\end{split}
\end{equation}
We write, for $\LL\geq 1$ and $\delta>0$,
\begin{equation}\label{DLL}
D(L,\LL):=\{l\in\ZZ[i]:\text{$\LL\leq|l|^2\leq 16\LL$, $l=1$ or $l_1l_2$ or $l_1^2l_2^2$ for some $l_1,l_2\in P(L)$}\},
\end{equation}
and for $P \in \FF(N)$ we consider the counting function
\begin{equation}\label{MPL}
M(P, L, \LL, \delta) := \sum_{l\in D(L,\LL)} \#\{\gamma \in R(l) : u(\gamma P, P) \leq \delta\}.
\end{equation}
By applying the bound \eqref{boundk}, grouping in \eqref{pretrace2} the possible $\gamma$'s into dyadic ranges
$\delta/2<u(\gamma P,P)\leq\delta$, and treating the range $0<u(\gamma P,P)\leq T^{-2}$ separately,
we arrive at the central inequality
\begin{equation}\label{11b}
|\phi_{j_0}(P )|^2  \preccurlyeq \sum_{\substack{k\in\ZZ\\\delta=2^k>T^{-2}}}
\min\left(T^2, \frac{T}{\sqrt{\delta}(1+\delta)^A}\right) \left(\frac{M(P, L, 1, \delta)}{L} + \frac{M(P, L, L^2, \delta)}{L^3} + \frac{M(P, L, L^4, \delta)}{L^4}\right).
\end{equation}
Here the implied constant depends on $A>2$, and the sum runs through all integral powers of $2$ exceeding $T^{-2}$, abbreviated as $\delta$. We note that the set $D(L,L^4)$ underlying $M(P, L, L^4, \delta)$ consists of perfect squares only.

We have now prepared the scene for the application of diophantine analysis. In particular, the upper bound \eqref{11b} is free of the choices of the amplifier and the test functions, and only depends on $P$, $T$, $L$ in a canonical way. At the heart of our argument is the estimation of $M(P, L, \LL, \delta)$ uniformly in $L$, $\LL$, $\delta$, the point $P\in\FF(N)$ lying in the fundamental domain given by \eqref{fund}, and the squarefree level $N\in\OO$. The rest of the paper is devoted to this diophantine counting problem.

\section{Diophantine inequalities}\label{DiophantineInequailitiesSection}

In this section, we summarize some useful properties of the matrices
\begin{equation}\label{matrices}
\gamma = \begin{pmatrix} a & b\\c & d\end{pmatrix}\in \M_2(\OO),\qquad\det \gamma = l,\qquad N \mid c,
\qquad u(\gamma P, P) \leq \delta,
\end{equation}
where $P = z+rj$ is a given point in the fundamental domain $\FF(N)$. Recall from \eqref{fund} and Lemma~\ref{lem3} that
\begin{equation}\label{sizes}
z \ll 1\qquad\text{and}\qquad r \gg |N|^{-1}.
\end{equation}

First, we observe that \eqref{udef} and \eqref{ratio} imply
\[\delta \geq u(\gamma P, P)\geq \frac{|\Im(\gamma P) - \Im(P)|^2}{2\Im(\gamma P)\Im(P)}
= \frac{1}{2} \left|\frac{|l|^{1/2}}{\| c P + d\|} - \frac{\| cP + d\|}{|l|^{1/2}}\right|^2 ,\]
and so
\begin{equation}\label{3}
\| c P + d\| = |l|^{1/2} (1 + O (\sqrt{\delta})).
\end{equation}
Similarly, from $u(\gamma^{-1}P,P)=u(P,\gamma P)\leq\delta$ we infer
\begin{equation}\label{5}
\| cP - a\| = |l|^{1/2} (1+ O(\sqrt{\delta})).
\end{equation}
From \eqref{3} and \eqref{5} we deduce
\begin{equation}\label{2}
rc,\ 2cz-a+d,\ a+d \ll |l|^{1/2} (1 + \sqrt{\delta}).
\end{equation}

Second, we explore an exact formula for $u(\gamma P,P)$ that follows from \eqref{action}, \eqref{action3}, \eqref{ratio}, \eqref{udef}, and its consequences on the condition that $u(\gamma P,P)\leqslant\delta$:
\begin{equation}
\label{BasicDiophantineConditions}
\delta \geq u(\gamma P, P) = \frac{\|a'P + b' - Pc'P - Pd'\|^2}{2r^2},
\qquad\begin{pmatrix}a'&b'\\c'&d'\end{pmatrix}:=\frac{1}{\sqrt{l}}\begin{pmatrix} a & b\\c & d\end{pmatrix}\in\SL_2(\CC).
\end{equation}
To avoid ambiguity of the square-root, we always
take $0\leq\arg(\sqrt{l})<\pi$ for the rest of the paper.
Note that for the special determinants $l\in D(L,\LL)$ considered in \eqref{DLL} we have $0\leq\arg(l)<\pi$ by \eqref{primeset}, hence in the sequel we can and we shall assume that its square-root satisfies $0\leq\arg(\sqrt{l})<\pi/2$. In particular, for a square $l=l_1^2l_2^2$ with $l_1,l_2\in P(L)$ we have $\sqrt{l}=l_1l_2$, again by \eqref{primeset}. We compute
\[Pc'P = - r^2 \ov{c'} + c' z^2 + \re(2rc'z)j.\]
The quaternion $a'P + b' - Pc'P - Pd'$ above has four components, and we consider the first two, the third, and the fourth separately. This gives
\begin{equation}\label{second}
r^2 \ov{c} \frac{l}{|l|} -c z^2 + (a-d)z+b \ll r |l|^{1/2}\sqrt{\delta},
\end{equation}
\begin{equation}\label{third}
\re\left(\frac{2cz-a+d}{\sqrt{l}}\right),\ \im\left(\frac{a+d}{\sqrt{l}}\right)\ll \sqrt{\delta}.
\end{equation}
Finally, we remark that \eqref{second} and \eqref{2} imply
\begin{equation}\label{50}
-c z^2 + (a-d)z + b \ll r |l|^{1/2}(1+\sqrt{\delta}).
\end{equation}

In the next five sections, we estimate $M(P,L,\LL,\delta)$ by bounding the number of quadruples $(a,b,c,d)$ satisfying the above relations \eqref{matrices}--\eqref{50}. We do this in several different ways, obtaining upper bounds that are strong in different ranges for $\LL$ and $\delta$.

\section{Counting I -- First bounds}

The previous section allows us to derive a simple first bound for $M(P,L,\LL,\delta)$. We recall that $l\in D(L,\LL)$ holds in the definition \eqref{MPL}, and that this implies $|l|\asymp\LL^{1/2}$.

First, we estimate the contribution $M_0(P, L, \LL, \delta)$ of the matrices $\gamma$ in \eqref{matrices} with $c=0$.
For fixed $l\in D(L,\LL)$ there are $\preccurlyeq 1$ choices for $a$ and $d$ by $ad = l$, and for each triple $(l,a,d)$ there are  $\ll 1 + |l|r^2\delta $ choices for $b$ by \eqref{second}. It follows that $M_0(P, L, 1, \delta)\preccurlyeq 1+r^2\delta$ and
\begin{equation}\label{M0bound}
\begin{split}
M_0(P, L, L^2, \delta) &\preccurlyeq \sum_{l \in D(L,L^2)} (1 + |l|r^2 \delta) \ll L^2(1+L r^2\delta),\\[4pt]
M_0(P, L, L^4, \delta) &\preccurlyeq \sum_{l \in D(L,L^4)} (1 + |l|r^2 \delta) \ll L^2(1+L^2 r^2\delta).
\end{split}\end{equation}
Now we estimate the remaining contribution of $\gamma$ with $c\neq 0$, which is
\begin{equation}\label{M3total}
M_1(P, L, \LL, \delta):=M(P, L, \LL, \delta)-M_0(P, L, \LL, \delta)
\end{equation}
The relations \eqref{matrices} and \eqref{2} show that the number of possibilities for $c\neq 0$ is
\begin{equation}\label{6}
\#c\,\ll\frac{\LL^{1/2}(1+\delta)}{(r|N|)^2}\qquad\text{excluding $c=0$}.
\end{equation}
For fixed $c$, the number of possibilities for $a$ and $d$ can be bounded similarly by \eqref{3} and \eqref{5}:
\begin{equation}\label{7}
\#a,\,\#d\,\ll\LL^{1/2}(1+\delta)\qquad\text{for fixed $c$}.
\end{equation}
For fixed $(c,a,d)$, the number of possibilities for $b$ is $\ll 1+r^2\LL^{1/2}(1+\delta)$ by \eqref{50}, hence by
\eqref{sizes}
\[M_1(P, L, \LL, \delta) \ll \frac{\LL^{3/2}(1+\delta)^3}{(r|N|)^2}\bigl(1+r^2\LL^{1/2}(1+\delta)\bigr)\ll \LL^2(1+\delta)^4.\]
Adding the earlier bounds on $M_0(P, L, \LL, \delta)$, it follows that
\begin{equation}\label{9}
M(P, L, \LL, \delta) \preccurlyeq \LL^2(1+\delta)^4+\LL^{3/2}r^2\delta.
\end{equation}

In the proofs of Theorems~\ref{thm1} to~\ref{thm3}, we shall assume that $r<T$, otherwise the statements follow from \eqref{finalphiPbound}.
Then in \eqref{11b} we choose $A>2$ sufficiently large in terms of $\eps$ to infer from $(TV)^{\eps}\leq L\leq TV$ and \eqref{9} that the contribution of $\delta>(TV)^\eps$ is negligible, namely $\ll_\eps (TV)^{-10}$. In particular, \eqref{11b} remains valid (up to a negligible error term that is
clearly absorbed in the right-hand side of \eqref{11} below, which by itself is $\gg T/V$) if we restrict the dyadic summation to $T^{-2}<\delta\preccurlyeq 1$. Using also \eqref{M0bound}, \eqref{M3total}, and $M(P, L, 1, \delta)\preccurlyeq 1+r^2\delta$ for $\delta\preccurlyeq 1$ (as follows from \eqref{9}), we conclude (for $r<T$)
\begin{equation}\label{11}
|\phi_{j_0}(P)|^2 \preccurlyeq \sup_{0<\delta\preccurlyeq 1}\min\left(T^2, \frac{T}{\sqrt{\delta}}\right)
\left(r^2\delta + \frac{1}{L} + \frac{M_1(P, L, L^2, \delta)}{L^3} + \frac{M_1(P, L, L^4, \delta)}{L^4}\right),
\end{equation}

In the next four sections, we estimate $M_1(P, L, L^2, \delta)$ and $M_1(P, L, L^4, \delta)$ for $0<\delta\preccurlyeq 1$.

\section{Counting II -- Parabolic matrices}
\label{ParabolicSection}

In this section, we study for $0<\delta\preccurlyeq 1$ the contribution to $M_1(P, L, \LL, \delta)$ of the matrices $\gamma$ in \eqref{matrices} with $c\neq 0$ and $(a-d)^2 + 4bc = 0$. We denote this contribution by $M_2(P, L, \LL, \delta)$, so that the remaining contribution of $\gamma$ with $c\neq 0$ and $(a-d)^2 + 4bc\neq 0$ equals
\begin{equation}\label{Mtotal}
M_3(P, L, \LL, \delta):=M_1(P, L, \LL, \delta)-M_2(P, L, \LL, \delta).
\end{equation}

Since $c\neq 0$ and $(a-d)^2 + 4bc=0$, we have that $a+d=\pm 2\sqrt{l}$, and so $\frac{1}{\sqrt{l}}\gamma\in\SL_2(\CC)$ is parabolic with unique fixed point $\frac{a-d}{2c}\in K$ (cf.\ \cite[p.~34]{EGM}). We write this fraction as $\frac{p}{q}$ with coprime $p,q\in\OO$, and follow the proof of \cite[Lemma~4.1]{HT1}. As $N$ is squarefree, $M := N/(N, q)$ is coprime to $q$, hence there exist $r,s\in\OO$ such that $Mps-qr = 1$. Then, by \eqref{normalizer}, the matrix $\sigma := \left(\begin{smallmatrix} Mp & r\\ Mq & Ms\end{smallmatrix}\right)$ lies in
$\Gamma_0^{\ast}(N)$ and maps $\infty$ to $\frac{a-d}{2c}$. Therefore the conjugated matrix
$\gamma' := \sigma^{-1} \gamma \sigma$ fixes $\infty$, hence by \eqref{normalizer} and \eqref{matrices} it is of the form
$\gamma' = \pm \left(\begin{smallmatrix}\sqrt{l} & b'\\ 0 & \sqrt{l}\end{smallmatrix}\right)\in\M_2(\OO)$. (Indeed, for any $\sigma \in \Gamma_0^{\ast}(N)$ and any $\gamma \in \M_2(\OO)$ whose lower left entry is  divisible by $N$, the matrix $\sigma^{-1} \gamma \sigma$ is integral.)
It follows that $b' \in \OO$ is a Gaussian integer, in addition to $l$ being a square that also follows from $a+d=\pm 2\sqrt{l}$.
Note that $b'\neq 0$, otherwise we would have $\gamma=\gamma'=\pm \left(\begin{smallmatrix}\sqrt{l} & 0\\ 0 & \sqrt{l}\end{smallmatrix}\right)$, contrary to the assumption $c\neq 0$. If we put $P' := \sigma^{-1} P$, then by $P\in\FF(N)$ and \eqref{fund} we have $\Im(P')\leq\Im(P)=r$, hence by \eqref{udef} and $b'\neq 0$
we infer
\[\delta \geq u(\gamma P, P) = u(\gamma' P', P') = \frac{|b'|^2}{2|l|\Im(P')^2} \geq \frac{|b'|^2}{2|l| r^2} \geq \frac{1}{8\LL^{1/2} r^2}.\]
To summarize so far,
\begin{equation}\label{M1bound1}
M_2(P, L, \LL, \delta)=0\qquad\text{unless}\qquad \delta\gg \LL^{-1/2}r^{-2}.
\end{equation}

We complement this observation with a general bound for $M_2(P, L, \LL, \delta)$. We can fix $c\neq 0$ in $\preccurlyeq\LL^{1/2}$ ways according to \eqref{6} and \eqref{sizes}, then $a$ in $\preccurlyeq\LL^{1/2}$ ways according to \eqref{7}. Now \eqref{5} implies $|l|-\|cP-a\|^2\preccurlyeq\LL^{1/2}\sqrt{\delta}$, hence there are $\preccurlyeq 1+\LL^{1/2}\sqrt{\delta}$ possible values for $|l|\in\ZZ$ (recall that $l\in\ZZ[i]$ is a square). Once we fix a value for this integer, $l$ itself is determined by its special shape. For a given triple $(c,a,l)$, there are two choices for the pair $(d,b)$ by $a+d=\pm 2\sqrt{l}$ and $ad-bc=l$. We conclude
\begin{equation}\label{M1bound2}
M_2(P, L, \LL, \delta)\preccurlyeq\LL(1+\LL^{1/2}\sqrt{\delta}).
\end{equation}

The relations \eqref{M1bound1} and \eqref{M1bound2} constitute our estimates on the contribution of parabolic matrices to $M_1(P,L,\LL,\delta)$. In the next three sections, we estimate $M_1(P, L, \LL, \delta)$ for $0<\delta\preccurlyeq 1$ either directly, or via \eqref{Mtotal}--\eqref{M1bound2} by focusing on the non-parabolic contribution $M_3(P, L, \LL, \delta)$.

\section{Counting III -- Large volume}

In this section, we derive bounds on $M_1(P, L, L^2, \delta)$ and $M_3(P, L, L^4, \delta)$ that are strong for $N$ large. We recall the assumption
$0<\delta\preccurlyeq 1$ and the fact that $|l|\asymp\LL^{1/2}$ holds for $l$ in the definition of $M(P, L, \LL, \delta)$.

First, we fix some $c\neq 0$ that occurs in \eqref{matrices}. Along with $c$, we also fix a matrix $\gamma'=\left(\begin{smallmatrix}a'&b'\\c&d'\end{smallmatrix}\right)$ satisfying \eqref{matrices}. Then, any matrix $\gamma=\left(\begin{smallmatrix}a&b\\c&d\end{smallmatrix}\right)$ satisfying \eqref{matrices} with the same $c$ is determined by the differences
\[a'' := a - a', \qquad d'' := d - d',  \qquad b'' := b - b'.\]
By \eqref{3}, \eqref{5}, \eqref{50}, these quantities satisfy the clean bounds
\[a''\preccurlyeq \LL^{1/4},\qquad d''\preccurlyeq \LL^{1/4},\qquad (a''-d'')z + b'' \preccurlyeq r \LL^{1/4}.\]
The first two of these conditions imply that $(a''-d'')r \preccurlyeq r \LL^{1/4}$ which, combined with the third one, yields
\[\| (a''-d'')P  + b''\| \preccurlyeq r \LL^{1/4}.\]

This estimate shows that the lattice point $(a-d)P+b\in\Lambda(P)$ has distance $\preccurlyeq r \LL^{1/4}$ from the fixed point $(a'-d')P+b'$. By Lemma~\ref{lem3}, the number of admissible pairs $(a-d, b)$, for a fixed $c\neq 0$, is at most
\[\preccurlyeq  1 +  r^2\LL^{1/2}|N| + r^2\LL.\]
Finally, we have $\preccurlyeq \LL^{1/2}$ choices for $a+d$ by \eqref{2}, and the quadruple $(c, a-d, b, a+d)$ determines the matrix $\gamma$ completely. We conclude, using also \eqref{6} and \eqref{sizes}, that
\begin{equation}\label{volume1}
M_1(P, L, L^2, \delta) \preccurlyeq \frac{L}{(r|N|)^2}L\left(1 +  r^2 L |N| + r^2 L^2\right) \ll L^2 + \frac{L^4}{|N|^2}.
\end{equation}

For $M_3(P, L, L^4, \delta)$ we argue similarly, utilizing also the facts that  $l$ is a square and $(a-d)^2 + 4bc\neq 0$ to count the final number of choices for $a+d$ more efficiently. Namely, the decomposition
\begin{equation}\label{traceformula}
0\neq (a-d)^2 + 4bc = (a+d)^2 - 4l = (a+d-2\sqrt{l})(a+d+2\sqrt{l})
\end{equation}
coupled with \eqref{2} shows that each triple $(c, a-d, b)$ gives rise to $\preccurlyeq 1$ choices for $a+d$.
We conclude that
\begin{equation}\label{volume2}
M_3(P, L, L^4, \delta) \preccurlyeq \frac{L^2}{(r|N|)^2} \left(1 +  r^2 L^2  |N| + r^2 L^4\right)\ll L^2 + \frac{L^6}{|N|^2}.
\end{equation}
In particular, using also \eqref{Mtotal} and \eqref{M1bound2},
\begin{equation}\label{volume2b}
M_1(P, L, L^4, \delta) \preccurlyeq L^6.
\end{equation}

Along with the results of Section~\ref{ParabolicSection}, the bounds \eqref{volume1} and \eqref{volume2} provide the necessary estimates in \eqref{11} to complete the proof of Theorem~\ref{thm1} in Section~\ref{EndgameSection}.

\section{Counting IV -- Small distances}

In this section, we develop a variant of the bounds \eqref{volume1} and \eqref{volume2} for $0<\delta\leq 1$ small. We fix an integer $0\leq n<2\pi/\sqrt{\delta}$, and count those matrices $\gamma$ entering $M_1(P,L,\LL,\delta)$ whose determinant $l$ satisfies
\begin{equation}\label{satisfies}
n\sqrt{\delta}\leq\arg(l)<(n+1)\sqrt{\delta},
\end{equation}
but is otherwise not fixed. The improvement from this subdivision comes from fine-tuning the argument to employ \eqref{second} in place of \eqref{50}.

First, we fix some $c\neq 0$ that occurs in \eqref{matrices}. Along with $n$ and $c$, we also fix a matrix $\gamma'=\left(\begin{smallmatrix}a'&b'\\c&d'\end{smallmatrix}\right)$ satisfying \eqref{matrices}. Then, any matrix $\gamma=\left(\begin{smallmatrix}a&b\\c&d\end{smallmatrix}\right)$ satisfying \eqref{matrices} with the same $c$ is determined by
the differences
\[a'' := a - a', \qquad d'' := d - d',  \qquad b'' := b - b'.\]
Moreover, by  \eqref{satisfies}  the determinants $l:=\det\gamma$ and $l':=\det\gamma'$ satisfy
\begin{equation}\label{lratio}
l/l'=x+O(\sqrt{\delta})\qquad\text{and}\qquad \sqrt{l}/\sqrt{l'}=\sqrt{x}+O(\sqrt{\delta})
\end{equation}
with $x:=|l/l'|\asymp 1$ real.
By \eqref{3}, \eqref{5}, \eqref{2}, \eqref{second}, the above differences satisfy the clean bounds
\begin{equation}\label{cleanbounds}
a''\ll \LL^{1/4},\qquad d''\ll \LL^{1/4},\qquad (a''-d'')z + b'' \ll r \LL^{1/4}\sqrt{\delta}.
\end{equation}
Moreover, from \eqref{2}, \eqref{third} and \eqref{lratio} we can derive that
\[\re\left(\frac{a''-d''}{\sqrt{l'}}\right),\  \im\left(\frac{a''+d''}{\sqrt{l'}}\right)\ll\sqrt{\delta}.\]
In particular, each of the lattice points $a-d=(a'-d')+(a''-d'')$ and $a+d=(a'+d')+(a''+d'')$ lies in a fixed rectangle of side lengths $\ll\LL^{1/4}\sqrt{\delta}$ and $\ll \LL^{1/4}$, whence the number of possibilities for it is $\ll \LL^{1/4}+\LL^{1/2}\sqrt{\delta}$ by a standard area argument. For a fixed $a-d$, we have $\ll 1+r^2\LL^{1/2}\delta$ choices for $b=b'+b''$ by \eqref{cleanbounds}, and the quadruple $(c, a-d, b, a+d)$ determines the matrix $\gamma$ completely.

Combining the contributions from the various angular sectors defined by \eqref{satisfies}, and  using also \eqref{6} and \eqref{sizes}, we conclude that
\[M_1(P, L, L^2, \delta)\ll\frac{1}{\sqrt{\delta}}\cdot\frac{L}{(r|N|)^2}(\sqrt{L}+L\sqrt{\delta})^2(1+r^2 L\delta)
\ll \frac{L^2}{\sqrt{\delta}}+L^4\delta^{3/2}.\]
The left hand side is non-decreasing in $\delta$, while the right hand side minimizes at $\delta\asymp L^{-1}$, hence for $\delta\leq L^{-1}$ we replace $\delta$ by $L^{-1}$ on the right hand side. This yields
\begin{equation}\label{volume3}
M_1(P, L, L^2, \delta)\ll\begin{cases}
L^{5/2}, & 0< \delta \leq L^{-1};\\
L^4\delta^{3/2}, & L^{-1} < \delta \leq 1.
\end{cases}
\end{equation}

For $M_3(P, L, L^4, \delta)$ we argue similarly, utilizing also the facts that  $l$ is a square and $(a-d)^2 + 4bc\neq 0$.
Specifically, \eqref{traceformula} coupled with \eqref{2} shows that each triple $(c, a-d, b)$ gives rise to $\preccurlyeq 1$ choices for $a+d$. We conclude that
\[M_3(P, L, L^4, \delta)\preccurlyeq\frac{1}{\sqrt{\delta}}\cdot\frac{L^2}{(r|N|)^2}(L+L^2\sqrt{\delta})(1+r^2 L^2\delta)
\ll \frac{L^3}{\sqrt{\delta}}+L^6\delta.\]
The left hand side is non-decreasing in $\delta$, while the right hand side minimizes at $\delta\asymp L^{-2}$, hence for $\delta\leq L^{-2}$ we replace $\delta$ by $L^{-2}$ on the right hand side. This yields, using also \eqref{Mtotal} and \eqref{M1bound2},
\begin{equation}\label{volume4}
M_1(P, L, L^4, \delta)\preccurlyeq\begin{cases}
L^4, & 0 < \delta \leq L^{-4};\\
L^6\delta^{1/2}, & L^{-4} < \delta \leq 1.
\end{cases}
\end{equation}

\section{Counting V -- Tiny distances}\label{TinyDistancesSection}

Although strong in certain ranges, the bounds \eqref{volume3} and \eqref{volume4} do not yet suffice to prove Theorems~\ref{thm2} and \ref{thm3}. In this section, we use a very different counting method that is superior for $0<\delta\leq 1$ very small. It is convenient to treat the middle range and the high range separately.

\subsection{Middle range} In order to bound $M_1(P, L, L^2, \delta)$, we recall that
$l\in D(L,L^2)$ implies that $l$ is of size $|l|\asymp L$, namely $l$ is the product of two split Gaussian primes from $P(L)$. In \eqref{matrices} we choose first $c\neq 0$ in $\ll L$ ways according to \eqref{6} and \eqref{sizes}, then $a$ in $\ll L$ ways according to \eqref{7}. Now \eqref{5} implies $|l|^2=\|cP-a\|^4+O(L^2\sqrt{\delta})$, hence there are $\ll 1+L^2\sqrt{\delta}$ possible values for $|l|^2\in\ZZ$. Once we fix a value for this integer, $l$ itself is determined by its special shape. Then $d$ is restricted to a disk of radius $\ll\sqrt{L\delta}$ by \eqref{third}, leaving $\ll 1+L\delta$ possible values for it. As the quadruple $(c, a, l, d)$ determines the matrix $\gamma$ completely, we have proved
\[M_1(P, L, L^2, \delta)\ll L^2(1+L^2\sqrt{\delta})(1+L\delta).\]
Combining this bound with \eqref{volume1} and \eqref{volume3}, we end up with
\begin{equation}\label{eig9}
M_1(P, L, L^2, \delta)\preccurlyeq\begin{cases}
L^2, & 0< \delta \leq L^{-4};\\
L^4\delta^{1/2}, & L^{-4} < \delta \leq L^{-3};\\
L^{5/2}, & L^{-3} < \delta \leq L^{-1};\\
L^4\delta^{3/2}, & L^{-1} < \delta \preccurlyeq 1.
\end{cases}
\end{equation}

\subsection{High range}\label{TinyDistancesHighRange} In order to bound $M_1(P, L, L^4, \delta)$, we recall that
$l\in D(L,L^4)$ implies that  $l$ is a square of size $|l|\asymp L^2$, namely
$\lambda:=\sqrt{l}$ is the product of two split Gaussian primes from $P(L)$. We
assume below that $0<\delta\leq\kappa L^{-4}$, where $\kappa>0$ is a sufficiently small absolute constant to make our claims valid; for example, we use without further comment that a bound $\ll\sqrt{\kappa}$ implies $<1/2$.

We observe first that the strong bound on $\delta$ forces that $a+d=n\lambda$ for some rational integer $n\in\ZZ$. By \eqref{third} the lattice triangle with vertices $0$, $\lambda$, $a+d$ has (a half-integral) area
\[\frac{1}{2}|\lambda|^2\left|\im\left(\frac{a+d}{\lambda}\right)\right|\ll L^2\sqrt{\delta}\leq\sqrt{\kappa},\]
whence\footnote{The same idea underlies Sublemma~2 of Koyama's paper~\cite{Ko}.} $a+d=x\lambda$ for some $x\in\RR$. We use arithmetic in $\ZZ[i]$ to see that $x$ equals a rational integer $n\in\ZZ$. Clearly, $x\in\RR\cap\QQ(i)=\QQ$, i.e. $x=n/m$ for some coprime $m,n\in\ZZ$ with $m>0$. Then $m(a+d)=n\lambda$ implies $m\mid\lambda$ in $\ZZ[i]$, so that $m=1$ by the special shape of $\lambda$. In short, $x=n$ as claimed. From $a+d=n\lambda$ and \eqref{2}, it follows that $n\in\ZZ$ is bounded. We fix the integer $n\ll 1$ for the rest of this section.

We choose a parameter $L\leq Q\leq L^2$, and apply Dirichlet's  approximation theorem (in a version for $\ZZ[i]$). This allows us to write
\[Nz = \frac{p}{q}  + O\left(\frac{1}{|q|Q}\right)\]
for suitable $p, q \in \ZZ[i]$ and $1 \leq |q| \leq Q$ and $(p, q) = 1$. We shall give two bounds for
$M_1(P, L, L^4, \delta)$, namely \eqref{eig5} and \eqref{eig6} below, the first to be used when $q$ is small (``$z$ well approximable"), the second to be used when $q$ is large (``$z$ badly approximable"). We write $c=c'N$ throughout, so that $c'\ll L$ by \eqref{sizes} and \eqref{2}.\\

For our first bound, we multiply the first part of \eqref{third} by the positive integer $|q\lambda|^2$ to rewrite it in the equivalent form
\[\re\left((2cz-a+d)q\cdot\ov{q\lambda}\right)\ll |q|^2L^2\sqrt{\delta}.\]
The idea behind this move is that $(2cz-a+d)q$ is well-approximated by $\xi:=2c'p-aq+dq\in\ZZ[i]$, namely
\begin{equation}\label{approximation}
(2cz-a+d)q=2c'(Nz)q-aq+dq=\xi+O(c'Q^{-1}).
\end{equation}
In particular, $\xi\ll |q|L$ follows from \eqref{2} and $c'\ll L$.
All these bounds imply that $\xi\ov{q\lambda}\in\ZZ[i]$ has real part $\ll|q|\Delta$, where
\begin{equation}\label{Deltadef}
\Delta:=|q|L^2\sqrt{\delta}+L^2 Q^{-1}.
\end{equation}
We note that $\Delta\geq 1$, because $Q\leq L^2$. In particular,
$\xi\ov{q\lambda}$ lies in a rectangle with side lengths $\ll|q|\Delta$ and $\ll|q|^2 L^2$, hence the Gaussian integer
$\xi\ov{\lambda}$ lies in a (not necessarily axis-parallel) rectangle with side lengths $\ll\Delta$ and $\ll|q|L^2$. By the usual area argument, it follows that there are $\ll |q|\Delta L^2$ possible values for the product $\xi\ov{\lambda}$. If this product is zero, then $\xi=0$, and we have $\ll L^2$ choices for the pair $(\xi,\lambda)$. Otherwise, the product determines
the pair $(\xi,\lambda)$ up to multiplicity $\preccurlyeq 1$ by the standard divisor bound. We conclude that
\[\#(\xi,\lambda)\preccurlyeq |q|\Delta L^2.\]

For fixed $\xi$ and $\lambda$, the matrix $\gamma$ is completely determined by $c'$ via the equations
\[2c'p-2aq=\xi-n\lambda q\qquad\text{and}\qquad 2c'p+2dq=\xi+n\lambda q,\]
hence it remains to estimate the number of choices for $c'$. Using the above two identities, the determinant equation $\lambda^2=ad-bc$ can be rewritten as
\[4\lambda^2 q^2=(2c'p-\xi+n\lambda q)(-2c'p+\xi+n\lambda q)-4bcq^2,\]
\[4c'(c'p^2-\xi p+Nbq^2)=(n^2-4)\lambda^2q^2-\xi^2.\]
If the right hand side is nonzero, then there are $\preccurlyeq 1$ choices for its divisor $c'$. We claim that the same conclusion also holds when the right hand side is zero. Indeed, in this case \eqref{approximation} implies that
\[(a-d)^2+4bc=(n^2-4)\lambda^2=(\xi/q)^2=\bigl(2cz-a+d+O( c' |q|^{-1}Q^{-1})\bigr)^2.\]
We square out on the right hand side, and compare the result with the left hand side. We obtain, using also that $2cz-a+d\ll L$ and $c'\ll L$,
\[-4(cz)^2+4(cz)(a-d)+4bc\ll \frac{Lc'}{|q|Q},\]
therefore
\[-cz^2+(a-d)z+b\ll\frac{L}{N|q|Q}.\]
This means that the first term in \eqref{second} is small, namely,
\[r^2 Nc'\ll \frac{L}{N|q|Q}+rL\sqrt{\delta},\]
and hence $c'\ll LQ^{-1}+L\sqrt{\delta}\ll 1$ by \eqref{sizes}, $L\leq Q$ and $\delta\leq\kappa L^{-4}$. So there are $\preccurlyeq 1$ choices for $c'$ simply because it is bounded. Altogether, we have proved that
\begin{equation}\label{eig5}
M_1(P, L, L^4, \delta) \preccurlyeq |q|\Delta L^2.
\end{equation}
This bound is strong when $|q|$ is small.\\

For our second bound, we estimate the number of possibilities for the pair $(c,a)$ more directly. Once this pair is chosen, \eqref{5} implies $|\lambda|^2=\|cP-a\|^2+O(L^2\sqrt{\delta})$, and hence the integer $|\lambda|^2$ is uniquely determined in the light of $\delta\leq\kappa L^{-4}$. As a result, $\lambda$ itself is determined by its special shape, and consequently the pair $(d,b)$ is determined by the equations $a+d=n\lambda$ and $ad-bc=\lambda^2$.

The basis of our estimations is equation \eqref{5}. It shows that,
for a given nonzero $c'\ll L$, the possible points $a$ lie in a disc of radius $\ll L$, and they also satisfy
\[\|c'p+cqrj-aq\|=\|cqP-aq\|+O(LQ^{-1})=|q\lambda|+O(|q|L\sqrt{\delta}+LQ^{-1}).\]
Squaring both sides and using $\lambda\ll L$, we obtain with the notation \eqref{Deltadef},
\[\re(2p\ov{q}c'\ov{a})=|q|^2(|a|^2-|\lambda|^2)+|c'p|^2+|cqr|^2+O(|q|\Delta).\]
For every $c'$ that occurs in our count, we fix one choice of corresponding $a_{c'}$ and $\lambda_{c'}$. Then the Gaussian integer $u:=c'(\ov{a}-\ov{a_{c'}})\ll L^2$ satisfies
\[p\ov{q}u+\ov{p}q\ov{u}=|q|^2(|a|^2-|a_{c'}|^2-|\lambda|^2+|\lambda_{c'}|^2)+O(|q|\Delta).\]
The left hand side lies in $\ZZ\cap(q,\ov{q})\ZZ[i]$, which is a lattice in $\RR$ of minimal length $\geq |(q,\ov{q})|$, so by $\Delta\geq 1$ this rational integer falls into $\ll |q|\Delta/|(q,\bar{q})$ residue classes modulo $|q|^2$, each lying in $(q,\bar{q})$.
Fixing one such residue class, the corresponding Gaussian integers $u$ lie in a translate of $\ker\theta$, where $\theta:\ZZ[i]\to\ZZ/|q|^2\ZZ$ is the homomorphism of additive groups given by $v\mapsto p\ov{q}v+\ov{p}q\ov{v}\pmod{|q|^2}$. The elements of $\ker\theta$ are divisible by $q/(q,p\ov{q})=q/(q,\ov{q})$. Moreover, $\imm\theta$ contains any residue class $t\mod{|q|^2}$ with $t\in\ZZ\cap 2(q,\ov{q})\ZZ[i]$. Indeed, $(q,\ov{q})=(p\ov{q},\ov{p}q,q\ov{q})$ shows that $t=2(p\ov{q}v_1+\ov{p}q v_2+q\ov{q} v_3)$ for some $v_1,v_2,v_3\in\ZZ[i]$, and then
\[t=(t+\ov{t})/2=p\ov{q}v+\ov{p}q\ov{v}+|q|^2 w\]
with $v:=v_1+\ov{v_2}$ a Gaussian integer and $w:=v_3+\ov{v_3}$ a rational integer. Identifying $\ZZ[i]$ with $\ZZ^2$,
we get that $\ker\theta$ is a lattice in $\RR^2$ of minimal length $\geq |q|/|(q,\ov{q})|$ and covolume
$[\ZZ[i]:\ker\theta]=\#\imm\theta\gg |q|^2/(q,\ov{q})$. So we conclude from Lemma~\ref{lem2} (with $n=2$, $R\ll L^2$) and $|q|\leq Q\leq L^2$ that
\[\#u \ll\frac{|q|\Delta}{|(q,\ov{q})|}\left(1+\frac{L^2|(q,\ov{q})|}{|q|}+\frac{L^4|(q,\ov{q})|}{|q|^2}\right)
\leq\Delta\left(|q|+L^2+\frac{L^4}{|q|}\right)\ll\frac{\Delta L^4}{|q|}.\]
This bounds the number of possibilities for $u = c'(\ov{a}-\ov{a_{c'}})$. If this product is zero, then
$a=a_{c'}$, and we have $\ll L^2$ choices for the pair $(c,a)$. Otherwise, the product determines the pair $(c,a)$ up to multiplicity
$\preccurlyeq 1$ by the standard divisor bound. Altogether, we have proved that
\begin{equation}\label{eig6}
M_1(P, L, L^4, \delta) \preccurlyeq L^2+\frac{\Delta L^4}{|q|}\ll\frac{\Delta L^4}{|q|}.
\end{equation}
This bound gives stronger results when $|q|$ is large.\\

We have established the bounds \eqref{eig5} and \eqref{eig6} under the assumptions $0<\delta\leq\kappa L^{-4}$ and $L\leq Q\leq L^2$.
For $|q|\leq L$ we use \eqref{eig5}, while for $|q|>L$ we use \eqref{eig6}. In both cases we obtain
\[M_1(P, L, L^4, \delta)\preccurlyeq L^6\sqrt{\delta}+L^5Q^{-1}.\]
Obviously it is best to choose $Q := L^2$  to minimize the right hand side. This yields
\[M_1(P, L, L^4, \delta)\preccurlyeq\begin{cases}
L^3, & 0< \delta \leq L^{-6};\\
L^{6}\sqrt{\delta} , & L^{-6}< \delta \leq \kappa L^{-4}.
\end{cases}\]

Combining this bound with \eqref{volume2b} and \eqref{volume4}, we end up with
\begin{equation}\label{eig8}
M_1(P, L, L^4, \delta)\preccurlyeq\begin{cases}
L^3, & 0< \delta \leq L^{-6};\\
L^6\sqrt{\delta}, & L^{-6} < \delta \preccurlyeq 1.
\end{cases}
\end{equation}

\section{The endgame}
\label{EndgameSection}

Theorems~\ref{thm1} to~\ref{thm3} are trivial when $TV$ is bounded, hence we can assume that $TV$ is sufficiently large in terms of $\eps$. Then the only condition on the amplifier length is the bound $(TV)^{\eps}\leq L\leq TV$, as specified in Section~\ref{amplifiedsection}.

\subsection{Proof of Theorem~\ref{thm1}} We assume $r\leq L^{-1}(TV)^{-\eps}$ to enforce $M_2(P, L, L^4, \delta)=0$ through \eqref{M1bound1}. Then \eqref{11}, \eqref{Mtotal}, \eqref{volume1}, \eqref{volume2} imply the clean bound
\[|\phi_{j_0}(P)|^2 \preccurlyeq T^2 \left(\frac{1}{L} + \frac{L^2}{|N|^{2}}\right).\]
This optimizes at $L := |N|^{2/3}(TV)^\eps$, where it furnishes
\[|\phi_{j_0}(P)|\preccurlyeq T|N|^{-1/3}\qquad\text{for}\qquad r\leq |N|^{-2/3}(TV)^{-2\eps}.\]
By Lemma~\ref{fourier} this bound remains true for $r > |N|^{-2/3}(TV)^{-2\eps}$, and we conclude Theorem~\ref{thm1}.

\subsection{Proof of Theorem~\ref{thm2}} By \eqref{eig9} and \eqref{eig8} we have the uniform bounds
\begin{align*}
M_1(P, L, L^2, \delta)&\preccurlyeq L^2+L^4\sqrt{\delta},\qquad 0<\delta\preccurlyeq 1;\\
M_1(P, L, L^4, \delta)&\preccurlyeq L^3+L^6\sqrt{\delta},\qquad 0<\delta\preccurlyeq 1.
\end{align*}
Inserting these into \eqref{11}, we obtain that
\[|\phi_{j_0}(P)|^2 \preccurlyeq \sup_{0<\delta\preccurlyeq 1}\min\left(T^2, \frac{T}{\sqrt{\delta}}\right)
\left(r^2\delta+\frac{1}{L}+L^2\sqrt{\delta}\right).\]
In particular,
\[|\phi_{j_0}(P)|^2  \preccurlyeq Tr^2 + T^2L^{-1}+ TL^2.\]
Upon choosing $L := T^{1/3}(TV)^\eps$, we arrive at
\[|\phi_{j_0}(P )|^2 \preccurlyeq Tr^2 + T^{5/3} \ll T^{5/3},\]
provided $r \leq T^{1/3}$. By Lemma~\ref{fourier} this bound remains true for $r > T^{1/3}$, and we conclude Theorem~\ref{thm2}.

\subsection{Proof of Theorem~\ref{thm3}} We interpolate between Theorem~\ref{thm1} and~\ref{thm2}. This gives
\[\| \phi \|_{\infty} \preccurlyeq T \min(V^{-1/6}, T^{-1/6}) \leq T  (V^{-1/6})^{1/3} ( T^{-1/6})^{2/3},\]
and Theorem~\ref{thm3} follows.

\end{document}